\documentclass[11pt]{article}

\usepackage[a4paper,margin=1in]{geometry}
\usepackage[T1]{fontenc}
\usepackage[utf8]{inputenc}
\usepackage{amsmath,amssymb,amsthm,mathtools}
\usepackage{enumitem}
\usepackage{hyperref}

\hypersetup{colorlinks=true,linkcolor=blue,citecolor=blue,urlcolor=blue}

\newtheorem{theorem}{Theorem}[section]
\newtheorem{proposition}[theorem]{Proposition}
\newtheorem{lemma}[theorem]{Lemma}

\newcommand{\Z}{\mathbb{Z}}

\newcommand{\w}{\omega}
\newcommand{\N}{\mathrm{N}}

\title{Infinite families of Diophantine quadruples in $\Z[\sqrt{-2}]$ in the remaining exceptional congruence classes}
\author{Andrej Dujella and Ivan Soldo}
\date{}

\begin{document}

\maketitle

\begin{abstract}
We continue the study of $D(z)$-quadruples in the ring $\Z[\sqrt{-2}]$. Motivated by the earlier classification due to the authors and by the subsequent partial results for the remaining families, we consider the exceptional congruence classes arising in the forms $24a+5+(12b+6)\sqrt{-2}$, $24a+2+(12b+6)\sqrt{-2}$, and $48a+44+(24b+12)\sqrt{-2}$. By combining the regular extension method with new families obtained by fixing a divisor $e\mid 3z$ and a small element $v\in \Z[\sqrt{-2}]$, we construct explicit $D(z)$-quadruples in each of the previously unsolved congruence classes. More precisely, we show that every exceptional class contains infinitely many values of $z$ admitting a twice semi-regular $D(z)$-quadruple, i.e., a quadruple containing two regular $D(z)$-triples. We also include remarks on the exceptional values $z\in\{-1,1\pm 2\sqrt{-2}\}$ and on a computational search in the exceptional congruence classes.

\end{abstract} 



\section{Introduction}

Let $R$ be a commutative ring with identity and let $z\in R$. A finite set $\{a_1,\dots,a_m\}\subset R\setminus\{0\}$ is called a $D(z)$-$m$-tuple if $a_i a_j+z$ is a square in $R$ for every $1\leq i<j\leq m$. The problem of determining whether such sets exist has a long history, beginning with the classical work on Diophantine tuples in $\Z$ and continuing in various rings of algebraic integers. Thus, the problem has been generalized in various directions (see \cite{DujeWeb} for an overview). Among imaginary quadratic rings, the case $R=\Z[\sqrt{-2}]$ is a relatively well-studied example. The first results were obtained by Abu Muriefah and Al-Rashed \cite{AbuMuriefahAlRashed2004}. Later, in \cite{DujellaSoldo2010, Soldo2013}, the authors gave an almost complete answer to the existence problem for $D(z)$-quadruples in this ring. Their main classification theorem can be stated as follows:

\begin{theorem}[{\cite[Theorem~1.1]{DujellaSoldo2010}}]\label{tm_Dujella_Soldo}
Let $z\in \Z[\sqrt{-2}]$. If $z$ is of the form
\[
 z=a+(2b+1)\sqrt{-2}
 \quad \text{or} \quad
 z=4a+(4b+2)\sqrt{-2},
 \quad a,b\in \Z,
\]
then there does not exist a $D(z)$-quadruple in $\Z[\sqrt{-2}]$. If $z$ is not of one of these two forms, then there exists at least one $D(z)$-quadruple, except possibly if $z$ has one of the forms
\[
 z=24a+2+(12b+6)\sqrt{-2},\quad
 z=24a+5+(12b+6)\sqrt{-2},
\]
\[
 z=48a+44+(24b+12)\sqrt{-2}, \quad a,b\in \Z,
\]
or if $z\in\{-1,1\pm 2\sqrt{-2}\}$.
\end{theorem}

The remaining three infinite families were then investigated in \cite{Soldo2013}. In that paper, explicit regular constructions involving elements of norms $11$ and $22$ were developed, and a large proportion of the exceptional residue classes modulo $11$ was settled. More precisely, the unsolved congruence classes were isolated in the following propositions of \cite{Soldo2013}: 
\begin{proposition} [{\cite[Proposition~3]{Soldo2013}}] \label{prop3_old}
If $z$ is of the form $z = 24a+5 + (12b+6) \sqrt{-2}$,
then there exists at least one Diophantine quadruple with the property $D(z)$,
for any $a, b \in \mathbb{Z}$, except maybe for $a \equiv a' \pmod{11}$, $b \equiv b' \pmod{11}$, where
\begin{eqnarray*}
(a', b') &\in& \{(0,3), (0,4), (0,6), (0,7), (1,1), (1,5), (1,9), (2,1), (2,2), (2,8),\\
&&  (2,9), (3,5), (4,3), (4,5),\!(4,7),\!(5,0), (5,2), (5,8), (5,10), (6,0),\\
&&  (6,5), (6,10), (7,2), (7,5),(7,8), (8,4), (8,5), (8,6), (9,1), (9,3),\\
&&  (9,7), (9,9), (10,0), (10,4), (10,6), (10,10)\}.
\end{eqnarray*}
\end{proposition} 

\begin{proposition} [{\cite[Proposition~4]{Soldo2013}}] \label{prop4_old}
If $z$ is of the form $z = 24a+2 + (12b+6) \sqrt{-2}$,
then there exists at least one Diophantine quadruple with the property $D(z)$,
for any $a, b \in \mathbb{Z}$, except maybe for $a \equiv a' \pmod{11}$, $b \equiv b' \pmod{11}$, where
\begin{eqnarray*}
(a', b') &\in& \{(0,3), (0,5), (0,7), (1,0), (1,2), (1,8), (1,10), (2,0), (2,5), (2,10),\\
&& (3,2), (3,5), (3,8), (4,4), (4,5), (4,6), (5,1), (5,3), (5,7), (5,9), (6,0),\\
&& (6,4), (6,6), (6,10), (7,3), (7,4), (7,6), (7,7), (8,1), (8,5), (8,9),\\
&&  (9,1), (9,2), (9,8), (9,9), (10,5)\}.
\end{eqnarray*}
\end{proposition}

\begin{proposition} [{\cite[Proposition~5]{Soldo2013}}] \label{prop5_old}
If $z$ is of the form $z = 48a+44 + (24b+12) \sqrt{-2}$,
then there exists at least one Diophantine quadruple with the property $D(z)$,
for any $a, b \in \mathbb{Z}$, except maybe for $a \equiv a' \pmod{11}$, $b \equiv b' \pmod{11}$, where
\begin{eqnarray*}
(a', b') &\in& \{(0,5), (1,1), (1,2), (1,8), (1,9), (2,1), (2,5), (2,9), (3,3), (3,4),\\
&& (3,6), (3,7), (4,0), (4,4), (4,6), (4,10), (5,1), (5,3), (5,7), (5,9),\\
&& (6,4), (6,5), (6,6), (7,2), (7,5), (7,8), (8,0), (8,5), (8,10), (9,0),\\
&& (9,2), (9,8), (9,10), (10,3), (10,5), (10,7)\}.
\end{eqnarray*}
\end{proposition}

Another useful ingredient from \cite{Soldo2013} is that in those remaining families certain elements of very small norm cannot occur in a $D(z)$-quadruple. Since we will use these restrictions only as background facts, we quote them here without proof. The following two statements are exactly the results from \cite[Section~2]{Soldo2013}.

\begin{proposition}[{\cite[Proposition~1]{Soldo2013}}]
Let $z=24a+5+(12b+6)\sqrt{-2}$, where $a,b\in \Z$. If
\[
s \in \{\pm 1, \pm 3, \pm 4,\\ \pm 1 \pm \sqrt{-2}, \pm 2\pm \sqrt{-2}, \pm 1\pm 2\sqrt{-2}\},
\]
then there does not exist a $D(z)$-quadruple of the form $\{s,t,u,v\}$ in $\Z[\sqrt{-2}]$.

\end{proposition}

\begin{proposition}[{\cite[Proposition~2]{Soldo2013}}]
Let $z=24a+2+(12b+6)\sqrt{-2}$, where $a,b\in \Z$. If
\[
s \in \{\pm 1, \pm 3, \pm 4,\\ \pm 1 \pm \sqrt{-2}, \pm 2\pm \sqrt{-2}, \pm 1\pm 2\sqrt{-2}\}, 
\]
then there does not exist a $D(z)$-quadruple of the form $\{s,t,u,v\}$ in $\Z[\sqrt{-2}]$.

\end{proposition}

The aim of the present paper is to further investigate the exceptional congruence classes from Propositions~\ref{prop3_old}, \ref{prop4_old} and \ref{prop5_old}.  This problem is also listed as open Problem~1.8 in \cite{Duje_Open}, which was intended as an addition to the book~\cite{DujeDioph_book}. By \cite[Theorem 1]{DujeFran}, each of $z$'s can be represented as a difference of two squares of elements in $\mathbb{Z}[\sqrt{-2}]$. In view of the analogous results 
known for certain quadratic fields (see, for example, \cite{DujeGlas,Fran2,Fran3,FranSoldo}), it is natural to expect that explicit examples should also exist for 
the remaining unsolved congruence classes. We first reconsider the family
$
 z=24a+5+(12b+6)\sqrt{-2},
$
which appears in Proposition~\ref{prop3_old}. Next we deal with the family
$
 z=24a+2+(12b+6)\sqrt{-2},
$
from Proposition~\ref{prop4_old}. Once this case is finished, the remaining family
$
 z=48a+44+(24b+12)\sqrt{-2}
$
from Proposition~\ref{prop5_old} follows immediately from a scaling argument via Lemma~\ref{lem:scaling} from the next section. Although the constructions obtained in the present paper yield $D(z)$-quadruples of the same general form as those in \cite{Soldo2013}, the viewpoint adopted here is slightly  different. Namely, there the search was organized by fixing a divisor $e\mid 3z$ and a small element $u$, while here it is more efficient to fix $e$ and a small element $v$ and then solve for the remaining parameters.

The final outcome of the present completion is summarized in the following theorem, which is the main result of this paper.

\begin{theorem}\label{thm:main}
Let $z\in \Z[\sqrt{-2}]$. In each exceptional congruence class appearing in Propositions~\ref{prop3_old}, \ref{prop4_old} and \ref{prop5_old}, there are infinitely many elements $z$ admitting
a $D(z)$-quadruple in~$\mathbb Z[\sqrt{-2}]$.
\end{theorem}
The proof rests on the partial existence results from Propositions~\ref{prop3_old}, \ref{prop4_old} and \ref{prop5_old}, 
the scaling argument from Lemma~\ref{lem:scaling}, and the new propositions established in Section~\ref{sec:new} of this paper. Together, these results provide 
infinite families of elements $z$ admitting  $D(z)$-quadruples in all exceptional congruence classes occurring in Propositions~\ref{prop3_old},
\ref{prop4_old} and \ref{prop5_old}. 

In Section~\ref{sec:comput}, we consider the range of applicability
of the construction method used in this paper by means of a computational
search in the exceptional congruence classes.  The search records
individual values $z$ for which this method did not produce a
$D(z)$-quadruple.  Moreover, for one representative value $z$, we prove
that no proper twice semi-regular $D(z)$-quadruple arises from this
construction.  This illustrates a limitation of the twice semi-regular
approach for certain individual values of $z$.

Finally, in Section~\ref{sec:individ}, we discuss the individual exceptional values
$z\in\{-1,1\pm2\omega\}$ from Theorem~1.1.  We give examples of
$D(z)$-triples for $z=1\pm2\omega$ and prove a non-extendibility result
for infinite family of such triples.


\section{Preliminaries}\label{sec:prelim}

In this section, we recall the method from \cite[Section 2]{DujellaSoldo2010} that will be used to construct new $D(z)$-quadruples in the exceptional cases. 
We also fix the notation and collect the auxiliary facts needed in the proofs of the main results.

Throughout the paper, for simplicity, we write
\[
\w=\sqrt{-2}, \quad \Z[\w]=\Z[\sqrt{-2}],
\]
and for $x=p+q\w\in \Z[\w]$ we denote by
\[
\N(x)=x\overline{x}=p^2+2q^2
\]
its norm.
As we mentioned, we shall repeatedly use the following elementary scaling lemma:

\begin{lemma}[{\cite[Lemma~1]{Soldo2013}}]\label{lem:scaling}
Let $\{a_1,a_2,a_3,a_4\}\subset \Z[\w]$ be a $D(z)$-quadruple and let $\xi\in \Z[\w]$. Then the set
\[
\{a_1\xi,a_2\xi,a_3\xi,a_4\xi\}
\]
has the property $D(z\xi^2)$ in $\Z[\w]$, while the set 
\[
\{\bar{a_1}, \bar{a_2}, \bar{a_3}, \bar{a_4}\}
\]
has the property $D(\bar{z})$ in $\Z[\w]$.
\end{lemma}

The method used in this paper is based on the standard regular extension procedure  and we shall refer to it as the regular construction. Suppose that $u,v,r\in \Z[\w]$ satisfy
\begin{equation}\label{eq:uvplusz}
uv+z=r^2.
\end{equation}
Then the set $\{u,v,u+v+2r\}$ is a regular $D(z)$-triple, because
\[u(u+v+2r)+z=(u+r)^2, \quad v(u+v+2r)+z=(v+r)^2.
\]
Applying the same idea once more, one finds that
\begin{equation}\label{eq:regular-quadruple}
\{u,v,u+v+2r,u+4v+4r\}
\end{equation}
is a $D(z)$-quadruple if and only if there exists $y\in \Z[\w]$ such that
\begin{equation}\label{eq:last-square}
u(u+4v+4r)+z=y^2.
\end{equation}
A straightforward computation transforms \eqref{eq:last-square} into
\[
3z=(u+2r-y)(u+2r+y).
\]
Hence, setting
\[
e=u+2r-y, \quad f=u+2r+y,
\]
we obtain
\begin{equation}\label{eq:ef}
ef=3z, \quad 2u+4r=e+f.
\end{equation}
Conversely, if $e,f\in \Z[\w]$ satisfy \eqref{eq:ef} and \eqref{eq:uvplusz}, then \eqref{eq:regular-quadruple} is a  $D(z)$-quadruple.

In \cite{Soldo2013}, the search was organized by fixing a divisor $e\mid 3z$, writing $f=3z/e$, and then looking for a small element $u$ such that
\begin{equation}\label{eq:r-from-eu}
r=\frac{e+f-2u}{4}
\end{equation}
belongs to $\Z[\w]$ and
\begin{equation}\label{eq:v-from-r-u}
v=\frac{r^2-z}{u}
\end{equation}
also lies in $\Z[\w]$. This yields the quadruple \eqref{eq:regular-quadruple} whenever the above divisibility conditions are satisfied.

For the present paper it is more convenient to reverse the point of view. We again fix a divisor $e\mid 3z$, but this time we also fix a small element $v\in \Z[\w]$ and regard $r$ as the main variable. Starting from \eqref{eq:uvplusz} and substituting $z=r^2-uv$ into the relation $2u+4r=e+3z/e$, we obtain
\[
2eu+4er=e^2+3r^2-3uv.
\]
After rearranging, this becomes
\begin{equation}\label{eq:u-from-r}
(2e+3v)u=3r^2-4er+e^2.
\end{equation}
Therefore, once $e$ and $v$ are fixed, any element $r\in \Z[\w]$ for which the right-hand side of \eqref{eq:u-from-r} is divisible by $2e+3v$ produces an element $u\in \Z[\w]$. The corresponding parameter is then recovered from
\begin{equation}\label{eq:z-from-r}
z=r^2-uv.
\end{equation}
By construction, the resulting set \eqref{eq:regular-quadruple} is a $D(z)$-quadruple. We shall refer to such a quadruple as a twice semi-regular $D(z)$-quadruple, because it contains two regular $D(z)$-triples, i.e., $\{u,v,u+v+2r\}$ and $\{v,u+v+2r, u+4v+4r\}$.

In this paper, this reformulation is especially useful when $e$ is chosen among small divisors of $3$ and $6$, and $v$ is taken from a short list of elements of small norm. Writing
\[
r=x+y\w, \quad x,y\in \Z,
\]
one may expand \eqref{eq:u-from-r} and \eqref{eq:z-from-r} explicitly and read off the corresponding values of $a$ and $b$ in the target families
\[
z=24a+5+(12b+6)\w, \quad z=24a+2+(12b+6)\w.
\]
Imposing the required divisibility conditions yields an explicit expression for $u$ in terms of $x$ and $y$, and hence a corresponding value of $z=r^2-uv$. 
Requiring this $z$ to be of one of the target forms leads to polynomial parametrizations $a=a(x,y)$ and $b=b(x,y)$ with $a,b\in\Z$, from which one obtains infinite subfamilies in prescribed residue classes modulo $11$.


For one residual congruence class in Proposition~\ref{prop4_old}, namely $(a,b)\equiv (2,5)\pmod{11}$, it is more efficient to return to the norm $19$ strategy already implicit in the earlier computations from \cite{DujellaSoldo2010, Soldo2013}. One fixes a small element $u$ with $\N(u)=19$ and then applies \eqref{eq:r-from-eu} and \eqref{eq:v-from-r-u} directly. This complementary family completes the treatment of the last unsolved class.

\section{New $D(z)$-quadruples for exceptional cases}\label{sec:new}
This section is devoted to constructing infinite families of twice semi-regular $D(z)$-quadruples
in the exceptional congruence classes from
Propositions~\ref{prop3_old}, \ref{prop4_old} and \ref{prop5_old}.
After establishing these results, Theorem~\ref{thm:main} follows
immediately. We use the method described in the previous section and
begin with the exceptional congruence classes appearing in Proposition~\ref{prop3_old}.
\begin{proposition} \label{prop3_new}

For each exceptional residue class 
$(a',b') \pmod {11}$ appearing in Proposition~\ref{prop3_old}, there exist
infinitely many pairs $a,b\in\mathbb Z$ with $a\equiv a' \pmod {11}$, $b\equiv b' \pmod {11}$,
such that, for $z=24a+5+(12b+6)\w$, there exists a $D(z)$-quadruple in $\Z[\w]$.
\end{proposition} 

\proof 
In each of the exceptional residue classes from
Proposition~\ref{prop3_old} we construct an infinite family of $z$'s for which a $D(z)$-quadruple exists.

Let  $z = 24a+5 + (12b+6) \w$.  First, set $e=1-\w$ and $v = 2 \w$. Then  
\begin{align*}
u&=\frac{1}{6} (x^2-2 y (4+y)+x (4+8 y)-2 \w (x+y) (x-2 (1+y))-3),\\
a&=\frac{1}{72} (-x^2+4 x (2+y)+2 y (4+y)-15),\\
b&=\frac{1}{36} (-x^2-2 x (2+y)+2 y (4+y)-15).
\end{align*}
Taking $x=-9$, $y=-18s+20$, $s \in \Z$, we obtain  
\begin{align*}
&u= -108 s^2+480 s-393+ (216 s^2-438 s + 187)\w,\\ 
&a=9 s^2-13s+1, \quad b=18 s^2-53 s +35.
\end{align*} 
For this choice we have $r= -9+(-18 s+20)\w$. Hence the corresponding value of $z$ is $z=216 s^2-312 s+29 +(216 s^2-636 s+ 426)\w$ and the associated $D(z)$-quadruple is
\begin{align*}
\{
&-108 s^2+480 s-393+ (216 s^2-438 s + 187)\w, 2 \w, \\
&-108s^{2}+480s-411+(216s^{2}-474s+229)\w,\\
&-108s^{2}+480s-429+(216s^{2}-510s+275)\w 
\}. 
\end{align*}
Reducing modulo $11$, one finds that
$
(a,b)\equiv (5,2), (0,4), (8,4)\pmod{11}
$
for
$
s\equiv 6, 8, 9\pmod{11},
$
respectively. 

If $x=12 s-21, y=-12 s-12, s \in \Z$, we have 
\begin{align*}
&u=-216s^{2}-12s+363+(396s+11)\w,\\ 
&a=-6s^{2}+21s+8, \quad b=12s^{2}+20s-19.
\end{align*}  
Similarly we get associated $D(z)$-quadruple. A reduction modulo $11$ shows that
$
(a,b)\equiv (4,3)$, $(6,5), (4,5)\pmod{11}
$
for
$
s\equiv 2, 6, 7\pmod{11},
$
respectively.

Now let
$
x=-12s-15, y=-6s+18, s\in\Z.
$
Then
\begin{align*}
u&=108s^{2}-36s-465+(-318s+53)\w,\\
a&=3 s^2-20 s-9, \quad b=-6 s^2-15 s+ 32.
\end{align*}
Hence we obtain another $D(z)$-quadruple. Reducing modulo $11$, we get 
$
(a,b)\equiv (1,1), (9,7)\pmod{11}
$
for
$
s\equiv 6, 7\pmod{11},
$
respectively.  

For the choice
$
x=39, y=-18s-10, s\in\Z,
$
the above formulas give
\begin{align*}
u&=-108s^{2}-1032s-261+(216s^{2}-6s-551)\w,\\
a&=9 s^2-31 s-37, \quad b=18 s^2+55 s-22.
\end{align*}
This yields one more $D(z)$-quadruple.  
Hence, modulo $11$, we obtain the residue class
$
(a,b)\equiv (5,8) \pmod{11}, 
$
when
$
s\equiv 8\pmod{11}.
$

If we take $e=1+\w, v = -2 \w$, we get 
\begin{align*}
u&=\frac{1}{6} (x^2+x (4-8 y)-2 (-4+y) y+2 \w (x-y) (-2+x+2 y)-3),\\
a&=\frac{1}{72} (-x^2-4 x (-2+y)+2 (-4+y) y-15),\\
b&=\frac{1}{36} (x^2-2 x (-2+y)-2 (-4+y) y-21).
\end{align*}
Substituting
$
x=12s-33, y=-6s, s\in\Z,
$
into the preceding formulas, we get
\begin{align*}
u&=108s^{2}-396s+159+(385-210s)\w,\\
a&=3 s^2+2 s-19, \quad b=6 s^2-33 s +26.
\end{align*}
For these values we get $D(z)$-quadruple.  
Reducing $a$ and $b$ modulo $11$, we obtain
$
(a,b)\equiv (8,6),\ (0,6)\pmod{11}
$
for
$
s\equiv 2, 9\pmod{11},
$
respectively.

Taking $x=12s-11$, $y=12s+2$, $s \in \Z$, we obtain  
\begin{align*}
&u= -216s^{2}+108s+43+(39-156s)\w,\\ 
&a=-6 s^2+11 s-2, \quad b=-12 s^2+3.
\end{align*} 
That parametrization yields another  $D(z)$-quadruple. 
Reducing modulo $11$, one finds that
$
(a,b)\equiv (9,3), (1,9) \pmod{11}
$
for
$
s\equiv 0, 4\pmod{11},
$
respectively. 

Choosing 
$
x=12s-9, y=12s+6, s\in\Z,
$
the formulas from above give
\begin{align*}
u&=-216s^{2}-12s+75-(180s+5)\w,\\
a&=-6 s^2+9 s+1, \quad b=-12 s^2-8 s+3.
\end{align*}
We get a corresponding $D(z)$-quadruple, and 
modulo $11$, obtain the residue class 
$ 
(a,b)\equiv (4,7) \pmod{11}
$ 
corresponding to
$ 
s\equiv 6\pmod{11}.
$ 

Now, set $e=1-\w$ and $v=-3-\w$. Then 
\begin{align*}
u&=\frac{1}{33} (-7 x^2-4 x (1+5 y)+2 y (16+7 y)+\w (3+5 x^2-2 y (2+5 y)-2 x (8+7 y))+9),\\
a&=\frac{1}{396} (x^2+2 x (5-8 y)-2 (-26+y) y-72),\\
b&=\frac{1}{99} (2 x^2+x (-13+y)+(5-4 y) y-45).
\end{align*} 
If $x=-66 s-26, y=66 s+28, s \in \Z$, we have 
\begin{align*}
&u=3564s^{2}+3072s+661+(1188s^{2}+936s+183)\w,\\ 
&a=165 s^2+141 s +30, \quad b=-132 s^2-104 s-21.
\end{align*}  
This construction provides another $D(z)$-quadruple. 
A reduction modulo $11$ shows that
$
(a,b)\equiv (2,8), (0,3), (9,9), (5,10), (3,5), (10,6)\pmod{11}
$
for
$
s\equiv 3, 4, 5,7, 8,10\pmod{11},
$
respectively. 

Taking $e=1+\w$ and $v=-3+\w$, we obtain 
\begin{align*}
u&=\frac{1}{33} (-7 x^2+4 x (-1+5 y)+2 y (-16+7 y)-\w (3+5 x^2+2 (2-5 y) y+2 x (-8+7 y))+9),\\
a&=\frac{1}{396} (x^2-2 y (26+y)+2 x (5+8 y)-72),\\
b&=\frac{1}{99}  (-2 x^2+x (13+y)+y (5+4 y)-54).
\end{align*} 

Let
$
x=-66s-26, y=-66s-28, s\in\Z.
$
Then
\begin{align*}
u&=3564s^{2}+3072s+661-(1188s^{2}+936s+183)\w,\\
a&=165 s^2+141 s +30, \quad b=132 s^2+104 s +20.
\end{align*}
In this way, we obtain another $D(z)$-quadruple. 
Reducing modulo $11$, we get 
$
(a,b)\equiv (2,2), (0,7),$ $(9,1), (5,0), (10,4)\pmod{11}
$
for
$
s\equiv 3, 4, 5, 7, 10\pmod{11},
$
respectively.   

Now we take $e=1-\w$ and $v=3+5\w$ and obtain 
\begin{align*}
u&=\frac{1}{153}(11 x^2-2 y (32+11 y)+4 x (5+13 y)+\w (-3-13 x^2+20 y+26 y^2+x (32+22 y))-21),\\
a&=\frac{1}{1836}(-5 x^2+2 y (98+5 y)+2 x (65+16 y)-366),\\
b&=\frac{1}{459}(-4 x^2-x (49+5 y)+y (65+8 y)-201).
\end{align*} 
If we choose $x=-96s+18$ and $y=-30s+33$, $s\in\Z$, then 
\begin{align*}
u&=1512s^{2}-1224s+57+(-216s^{2}-600s+251)\w,\\
a&=30 s^2-76 s+ 20, \quad b=-96 s^2+42 s +12.
\end{align*}
Therefore, we have another $D(z)$-quadruple for the corresponding value of $z$. 
Modulo $11$, we get the residue class 
$ 
(a,b)\equiv (7,2) \pmod{11}
$ 
corresponding to
$ 
s\equiv 1\pmod{11}.
$ 

Further, taking $x=-66s-48$ and $y=-78s-45$, $s\in\Z$, yields 
\begin{align*}
u&=1188s^{2}+1752s+621+(1404s^{2}+1596s+443)\w,\\
a&=111 s^2+125 s +34, \quad b=12 s^2-10 s-10.
\end{align*}
Thus the above choice yields another $D(z)$-quadruple in the required family. 
A reduction modulo $11$ shows that
$
(a,b)\equiv (2,9), (6,10)\pmod{11}
$
for
$
s\equiv 5, 6\pmod{11},
$
respectively.

Suppose that $e=1+\w$ and $v=3-5\w$. Then we obtain 
\begin{align*}
u&=\frac{1}{153} (x (20+11 x)+64 y-52 x y-22 y^2+\w (3+x (-32+13 x)+20 y+22 x y-26 y^2)-21),\\
a&=\frac{1}{1836} (-5 x^2+x (130-32 y)+2 y (-98+5 y)-366),\\
b&=\frac{1}{459}(4 x^2+x (49-5 y)+(65-8 y) y-258).
\end{align*} 

For $x=-96s-62$ and $y=30s-25$, $s\in\Z$, it follows 
\begin{align*}
u&=1512s^{2}+888s-359+(216s^{2}+1368s+453)\w,\\
a&=30 s^2-60 s-36, \quad b=96 s^2+118 s-5
\end{align*}
and we get a corresponding $D(z)$-quadruple. A reduction modulo $11$ shows that
$
(a,b)\equiv (2,1), (6,0) \pmod{11}
$
for
$
s\equiv 6, 9\pmod{11},
$
respectively. 

Choosing $x=-96s-54$ and $y=30s-5$, $s\in\Z$, we obtain 
\begin{align*}
u&=1512s^{2}+1176s+105+(216s^{2}+792s+293)\w,\\
a&=30 s^2-20 s-16, \quad b=96 s^2+102 s+15. 
\end{align*} 
Then it is easy to obtain another $D(z)$-quadruple. Modulo $11$, this gives the additional residue class
$
(a,b)\equiv (7,8) \pmod{11}
$
corresponding to
$
s\equiv 3 \pmod{11}.
$

Let  $e=1-\w$ and $v=-3-7\w$. It follows that 
\begin{align*}
u&=\frac{1}{369} (-7 x^2+2 y (40+7 y)-4 x (13+23 y)+\w (-3+23 x^2-2 x (20+7 y)-2 y (26+23 y))+33),\\
a&=\frac{1}{4428} (x (202+13 x)+y(484 -40 x -26 y)-852),\\
b&=\frac{1}{1107} (5 x^2+(101-10 y) y+x (-121+13 y)-498).
\end{align*} 
For the choice
$
x=-120s+76, y=-78s+58, s\in\Z,
$
the above formulas give
\begin{align*}
u&=-2376s^{2}+3216s-1079+(-216s^{2}+504s-243)\w,\\
a&=-78 s^2+102 s-33, \quad b=120 s^2-146 s +44, 
\end{align*}
so we can construct a $D(z)$-quadruple. Reducing modulo $11$, one finds that
$
(a,b)\equiv (7,5)\pmod{11}
$
for
$
s\equiv 4\pmod{11}.
$

Taking $
x=-78s-60, y=60s-96, s\in\Z
$,
we obtain 
\begin{align*}
u&=1188s^{2}-1560s-1167+(108s^{2}+1872s-1123)\w,\\
a&=39 s^2+63 s-109, \quad b=-60 s^2+206 s-2. 
\end{align*}
Therefore, we have a $D(z)$-quadruple and a reduction modulo $11$ shows that 
$
(a,b)\equiv (8,5)\pmod{11}
$
for
$
s\equiv 4\pmod{11}.
$

Now we choose  $e=1-\w$, $v=-9-7\w$, and obtain 
\begin{align*}
u&=\frac{1}{561} (-25 x^2-4 x (7+23 y)+2 y (64+25 y)+\w (9+23 x^2-2 y (14+23 y)-2 x (32+25 y))+39),\\
a&=\frac{1}{6732} (7 x (46+x)+ y(772 -64 x -14 y)-1290),\\
b&=\frac{1}{1683} (-193 x+8 x^2+7 (23+x) y-16 y^2-753).
\end{align*}  
Substituting $x=-150s-30$, $y=-138s+69$, $s \in \Z$, we get 
\begin{align*}
u&=-2700s^{2}-1104s+741+(-2484s^{2}+2508s-169)\w,\\
a&=-213 s^2+85 s+ 17, \quad b=12 s^2+202 s-40, 
\end{align*}
and we can construct one more $D(z)$-quadruple. Modulo $11$, we obtain the residue class 
$
(a,b)\equiv (1,5)\pmod{11}
$
corresponding to
$
s\equiv 5\pmod{11}.
$

\noindent
Similarly, taking $x=-150s-44$, $y=-138s-5$, $s \in \Z$, it follows that 
\begin{align*}
u&=-2700s^{2}-1608s-119+(-2484s^{2}-156s+63)\w,\\
a&=-213 s^2-77 s-3, \quad b=12 s^2+82 s +14. 
\end{align*}
That leads to another $D(z)$-quadruple. Further, additional residue class is $
(a,b)\equiv (10,10)\pmod{11}
$
for 
$
s\equiv 7\pmod{11}.
$

Finally, let us take $e=1+\w$, $v=-9+7\w$. Then 
\begin{align*}
u&=\frac{1}{561} (-25 x^2+4 x (-7+23 y)+2 y (-64+25 y)-\w (9+x (-64+23 x)+28 y+50 x y-46 y^2)+39),\\
a&=\frac{1}{6732} (7 x^2+x (322+64 y)-2 (645+y (386+7 y))),\\
b&=\frac{1}{1683} (193 x-8 x^2+7 (23+x) y+16 y^2-930).
\end{align*}  
For the choice $x=-150s+6$, $y=138s-41$, $s \in \Z$, the above formulas give
\begin{align*}
u&= -2700s^{2}+192s+117+(2484s^{2}-1500s+161)\w,\\
a&=-213 s^2+65 s-1, \quad b=-12 s^2-74 s+11. 
\end{align*} 
We obtain the corresponding $D(z)$-quadruple. Reducing $a$ and $b$ modulo $11$, we get 
$
(a,b)\equiv (10,0)\pmod{11}
$
when
$
s\equiv 0\pmod{11}.
$
\endproof

Next, we consider the exceptional congruence classes from Proposition~\ref{prop4_old}. These are again treated by explicit regular constructions. 
As before, we present the computation in detail for one representative choice of 
parameters, while for the others we record only the corresponding formulas and congruence information.
\begin{proposition} \label{prop4_new}

For each exceptional residue class 
$(a',b') \pmod {11}$ appearing in Proposition~\ref{prop4_old}, there exist
infinitely many pairs $a,b\in\mathbb Z$ with $a\equiv a' \pmod {11}$, $b\equiv b' \pmod {11}$, 
such that, for $z=24a+2+(12b+6)\w$, there exists a $D(z)$-quadruple in $\Z[\w]$.
\end{proposition}  
\proof

Similarly, for each exceptional congruence class modulo $11$ appearing
in Proposition~\ref{prop4_old}, we construct an explicit family of
values $z$ admitting a twice semi-regular $D(z)$-quadruple.

Let $z = 24a+2 + (12b+6) \w$. We begin with the choice $e=2+\w$ and $v=-1-3\w$. This yields
\begin{align*}
u&=\frac{1}{33} (x^2+4 x (4-7 y)+\w (6+7 x^2+2 x (-10+y)+2 (8-7 y) y)-2 (9+(-20+y) y)),\\
a&=\frac{1}{99} (-x^2+x (17-5 y)+(-5+y) (3+2 y)),\\
b&=\frac{1}{198} (5 x^2+x (14-8 y)+2 (34-5 y) y-123).
\end{align*}
Taking
$
x=12s-7, y=15s+21, s\in\Z,
$
the preceding formulas become
\begin{align*}
u&=-162s^{2}-144s+121+(-54s^{2}-294s-171)\w,\\
a&=-6 s^2+8 s +13,\quad b=-15 s^2-36 s-9.
\end{align*}
The corresponding value of $r$ is
$
r=12s-7+(15s+21)\w,
$
and therefore
$
z=-144s^{2}+192s+314+(-180s^{2}-432s-102)\w.
$
In this way we obtain the $D(z)$-quadruple
\begin{align*}
\{
&-162s^{2}-144s+121+(-54s^{2}-294s-171)\w, -1-3\w, \\
&-162s^{2}-120s+106+(-54s^{2}-264s-132)\w,\\
&-162s^{2}-96s+89+(-54s^{2}-234s-99)\w 
\}. 
\end{align*}
A reduction modulo $11$ shows that
$
(a,b)\equiv (4,6), (5,1), (10,5), (1,8) \pmod{11}
$
for
$
s\equiv 1,3,6,8 \pmod{11},
$
respectively. 

Let $x=30s+13, y=-12s+34,  s\in\Z$. Then
\begin{align*}
u&=324s^{2}-660s-393+(108s^{2}+540s-419)\w,\\
a&=12 s^2-62 s-1,\quad b=30 s^2+24 s-60.
\end{align*}
In this way, we obtain another $D(z)$-quadruple.  
Reducing modulo $11$, we get $(a,b)\equiv (4,5), (0,7), (7,6) \pmod{11}$ for $s\equiv 1, 5, 8 \pmod{11}$, respectively. 

Take $x=-30s-43, \ y=12s-24, \ s\in\Z$. Then
\begin{align*}
u&=324s^{2}-60s-905+(108s^{2}+828s+225)\w,\\
a&=12 s^2-54 s-65,\quad b=30 s^2+88 s-36.
\end{align*}
This yields one more $D(z)$-quadruple. 
A reduction modulo $11$ shows that $(a,b)\equiv (3,5), (7,7), (7,3) \pmod{11}$ for $s\equiv 1, 2, 8\pmod{11}$, respectively. 

Substituting $x=-18s+15, y=27s-12, s\in\Z$ into the above formulas, we obtain
\begin{align*}
u&=378s^{2}-480s+143+(-270s^{2}+222s-39)\w,\\
a&=36 s^2-44 s +13,\quad b=-9 s^2+2 s+2.
\end{align*}
This gives another $D(z)$-quadruple. 
Reducing $a$ and $b$ modulo $11$, we obtain $(a,b)\equiv (7,4), (6,4) \pmod{11}$ for $s\equiv 3, 7 \pmod{11}$, respectively. 

For $x=30s+31, \ y=-12s-25, \ s\in\Z$, we get
\begin{align*}
u&=324s^{2}+972s+633+(108s^{2}+48s-139)\w,\\
a&=12 s^2+56 s+ 49,\quad b=30 s^2+60 s +17.
\end{align*}
Hence we obtain another $D(z)$-quadruple. 
Modulo $11$, this gives $(a,b)\equiv (2,0), (5,9) \pmod{11}$ for $s\equiv 5, 10 \pmod{11}$, respectively.

Choose $x=30s-31, y=-12s-9, s\in\Z$. Then
\begin{align*}
u&=324s^{2}-156s-239+(108s^{2}-504s+201)\w,\\
a&=12 s^2+24 s-27,\quad b=30 s^2-64 s+3.
\end{align*}
The corresponding construction gives another $D(z)$-quadruple. 
Reducing modulo $11$, one finds that $(a,b)\equiv (9,2), (3,2) \pmod{11}$ for $s\equiv 1, 7 \pmod{11}$, respectively. 

Let $x=-30s-15, y=12s,  s\in\Z$. For these values,
\begin{align*}
u&=324s^{2}+180s-1+(108s^{2}+204s+57)\w,\\
a&=12 s^2-6 s-5,\quad b=30 s^2+32 s+4.
\end{align*}
This leads to one more $D(z)$-quadruple. 
A reduction modulo $11$ shows that $(a,b)\equiv (1,0), (9,1)$, $(6,0), (0,5) \pmod{11}$ for $s\equiv 1, 2, 6, 9 \pmod{11}$, respectively. 

Now let $x=18s-45, y=-27s+12, s\in\Z$. Then
\begin{align*}
u&=378s^{2}-1248s+503+(-270s^{2}-6s+369)\w,\\
a&=36 s^2-64 s+1,\quad b=-9 s^2-74 s+66.
\end{align*}
In this way, we obtain another $D(z)$-quadruple. 
Moreover, $(a,b)\equiv (8,9), (2,10) \pmod{11}$ for $s\equiv 7, 10 \pmod{11}$, respectively. 

For the choice $x=12s-33,  y=15s-18,  s\in\Z$, the above expressions simplify to
\begin{align*}
u&=-162s^{2}+636s-529+(-54s^{2}+18s+141)\w,\\
a&=-6 s^2+34 s-39,\quad b=-15 s^2+42 s-22.
\end{align*}
This produces another  $D(z)$-quadruple. 
Reducing modulo $11$, we obtain the residue classes $(a,b)\equiv (1,10), (9,9) \pmod{11}$ for $s\equiv 9, 10 \pmod{11}$, respectively.

Now let us take $e=2-\w$ and $v=-1 + 3 \w$. We get 
\begin{align*}
u&=\frac{1}{33} (x^2+4 x (4+7 y)-2 (9+y (20+y))+\w (-6-7 x^2+2 x (10+y)+2 y (8+7 y))),\\
a&=\frac{1}{99} (-x^2+(5+y) (-3+2 y)+x (17+5 y)),\\
b&=\frac{1}{198} (-5 x^2-2 x (7+4 y)+2 y (34+5 y)-75).
\end{align*}
First, take $x=18s+25, y=27s+5, s\in\Z$. Then we obtain 
\begin{align*}
u&=378s^{2}+636s+129+(270s^{2}-6s-97)\w,\\
a&=36 s^2+40 s+5,\quad b=9 s^2-32 s-20.
\end{align*}
Thus we get another $D(z)$-quadruple, and modulo $11$ this covers the residue classes $(a,b)\equiv (1,2), (0,3)\pmod{11} $ for $s\equiv 6, 8 \pmod{11}$, respectively. 

Let $x=12s+25, y=-15s+5, s\in\Z$. Then
\begin{align*}
u&=-162s^{2}-216s+129+(54s^{2}-210s-97)\w,\\
a&=-6 s^2-24 s+5,\quad b=15 s^2-16 s-20.
\end{align*}
We arrive at another $D(z)$-quadruple. 
Reducing modulo $11$, one obtains $(a,b)\equiv (8,1) \pmod{11}$ for $s\equiv 1 \pmod{11}$, respectively. 

Choose $x=12s+27, y=-15s-3, s\in\Z$. This gives
\begin{align*}
u&=-162s^{2}-336s-31+(54s^{2}-126s-141)\w,\\
a&=-6 s^2-26 s-7,\quad b=15 s^2-18.
\end{align*}
Hence this provides another $D(z)$-quadruple. 
Modulo $11$, we get $(a,b)\equiv (4,4) \pmod{11}$ for $s\equiv 0 \pmod{11}$, respectively. 

We next take $x=-12s+25, y=15s-17, s\in\Z$. The corresponding values are
\begin{align*}
u&=-162s^{2}+480s-327+(54s^{2}-54s-29)\w,\\
a&=-6 s^2+24 s-19,\quad b=15 s^2-28 s+8.
\end{align*}
This yields another $D(z)$-quadruple in the required family.  
A reduction modulo $11$ shows that $(a,b)\equiv (3,8), (9,8) \pmod{11}$ for $s\equiv 0, 7 \pmod{11}$, respectively. 

Now, let $x=30s+17,  y=12s-6,  s\in\Z$. Then
\begin{align*}
u&=324s^{2}+60s-65+(-108s^{2}-252s-45)\w,\\
a&=12 s^2-6 s-5,\quad b=-30 s^2-32 s-5.
\end{align*}
In this way, we obtain another $D(z)$-quadruple. 
Reducing modulo $11$, we get $(a,b)\equiv (6,6), (6,10) \pmod{11}$ for $s\equiv 0, 6 \pmod{11}$, respectively. 

Further, set $e=2+\w$ and $v=2$. We have 
\begin{align*}
u&=\frac{1}{18} (5 x^2+4 x (-4+y)+2 (4-5 y) y+\w (6-x^2+2 (-8+y) y+2 x (-2+5 y))+6), \\
a&=\frac{1}{54} (x^2-x (-4+y)-2 (3+y+y^2)),\\
b&=\frac{1}{108} (4 x+x^2+8 (2+x) y-2 y^2-60).
\end{align*}
Take $x=-18s+24, y=18s-18, \ s\in\Z$. Then
\begin{align*}
u&=-162s^{2}+312s-145+(-162s^{2}+384s-225)\w,\\
a&=-8 s+9,\quad b=-27 s^2+62 s-35.
\end{align*}
This yields another $D(z)$-quadruple. 
A reduction modulo $11$ shows that $(a,b)\equiv (5,3) \pmod{11}$ for $s\equiv 6 \pmod{11}$, respectively. 

Take now  $e=2+\w$ and $v=5-3\w$. We obtain 
\begin{align*}
u&=\frac{1}{153} (19 x^2+88 y-38 y^2-4 x (8+7 y)+\w (7 x^2-2 (3+y) (-5+7 y)+x (-44+38 y))-6),\\
a&=\frac{1}{459} (2 x^2+x (53-11 y)-(3+y) (19+4 y)),\\
b&=\frac{1}{918} (62 x+11 x^2+4 (53+4 x) y-22 y^2-543).
\end{align*}
Substituting $x=24s-5,  y=-33s-17,  s\in\Z$ into the above formulas, we obtain
\begin{align*}
u&=-54s^{2}-288s-93+(-270s^{2}-174s+1)\w,\\
a&=12 s^2-4,\quad b=-33 s^2-40 s-10.
\end{align*}
This gives another $D(z)$-quadruple. 
Reducing $a$ and $b$ modulo $11$, we obtain $(a,b)\equiv (8,5) \pmod{11}$ for $s\equiv 1 \pmod{11}$, respectively. 

Further, if  $x=-24s+23,  y=33s+21,  s\in\Z$, yields 
\begin{align*}
u&=-54s^{2}-504s-125+(-270s^{2}-114s+93)\w,\\
a&=12 s^2-28 s-12,\quad b=-33 s^2-36 s+10.
\end{align*}
Hence we obtain another $D(z)$-quadruple. 
Modulo $11$, this gives $(a,b)\equiv (5,7) \pmod{11}$ for $s\equiv 1 \pmod{11}$, respectively.

The only residue class from Proposition~\ref{prop4_old} not covered by the preceding fixed $v$ constructions is $(2,5)$. 
To deal with this final case, we revert to the approach from \cite{Soldo2013} in which one fixes $u$ and a divisor $e\mid 3z$. Fix
$
u=1+3\w,$ and $e=2+\w.
$
From \eqref{eq:ef}, for $z=24a+2+(12b+6)\w$ we obtain
$
r=(6a+3b+2)+3(b-a)\w.
$
Furthermore, from \eqref{eq:v-from-r-u} we get
\[
v=
-\frac{198a^2-180ab+72a-99b^2-12b+34}{19}
-\frac{3(30a^2+66ab+4a-15b^2+12b+4)}{19}\w.
\]
Thus $v\in\Z[\w]$ whenever both numerators are divisible by $19$. 
For the residue class $(a,b)\equiv(2,5)\pmod{11}$, this is achieved by taking 
$a=209s+13$, $b=209s+5$, $s\in\Z$. 
Then
\begin{align*}
r&=1881s+95-24\w,\\
z&=5016s+314+(2508s+66)\w,\\ 
v&=186219s^2-10758s-1063+(-558657s^2-60522s-1437)\w, 
\end{align*}
and we obtain the corresponding $D(z)$-quadruple. 
\endproof

As we mentioned, Proposition~\ref{prop4_new} immediately implies 
\begin{proposition} \label{prop5_new}
For each exceptional residue class $(a',b') \pmod {11}$ appearing in Proposition~\ref{prop5_old}, there exist
infinitely many pairs $a,b\in\mathbb Z$ with $a\equiv a' \pmod {11}$, $b\equiv b' \pmod {11}$, such that, for
$ z=48a+44+(24b+12)\w$, 
there exists a $D(z)$-quadruple in $\Z[\w]$.
\end{proposition}

Namely, taking $\xi=\w$ in Lemma~\ref{lem:scaling}, every $D(z_0)$-quadruple gives rise to a $D(z)$-quadruple, where
$
z=48(-a-1)+44+(24(-b-1)+12)\w. 
$
Thus the parametrized families from Proposition~\ref{prop4_new} yield parametrized
families in the corresponding exceptional classes from Proposition~\ref{prop5_old}. 
Checking the residues $(-a-1,-b-1)\pmod{11}$ for the pairs $(a,b)$ obtained, as functions of the parameter $s$, 
in the proof of Proposition~\ref{prop4_new} gives all exceptional classes
listed in Proposition~\ref{prop5_old}.


\section{A computational search in the exceptional congruence classes}\label{sec:comput}
We performed a computational search in order to test the extent to
which the regular constructions considered in this paper cover the
remaining exceptional congruence classes. More precisely, the search was
carried out for all representatives satisfying
$-100\leq a\leq 100$ and $-100\leq b\leq 100$ in the exceptional
congruence classes modulo $11$.

The search was based on the regular construction described earlier, with
the aim of finding twice semi-regular $D(z)$-quadruples. For
$u=p+q\w$, we first used the bound $|p|,|q|\leq 5000$, considering
all $e=x+y\w$ with $|x|,|y|\leq 50$. For the cases that still remained,
we then let $e$ run through all divisors of $3z$ and repeated the search
with $|p|,|q|\leq 5000$. Finally, for the remaining cases up to
conjugation, we also performed an extended search with
$|p|,|q|\leq 10000$ and $|x|,|y|\leq 100$, whenever this computation was
feasible within the available running time.

Here, a tested pair $(e,u)$ means a choice of an admissible divisor
$e\mid 3z$ in the prescribed range, together with a value
$u=p+q\w$ satisfying the corresponding bound, for which the regular
construction was checked. In the extended search with
$|p|,|q|\leq 10000$ and $|x|,|y|\leq 100$, the number of tested pairs
$(e,u)$ ranged from $8\cdot 10^8$ to $2.2\cdot 10^9$ for the completed
cases, depending on the number of admissible divisors $e$.

By Lemma~\ref{lem:scaling}, it is enough to consider the remaining
cases up to conjugation. For
\[
 z=24a+c+(12b+6)\w,\quad c\in\{2,5\},
\]
we have
\[
 \overline z
 =24a+c-(12b+6)\w
 =24a+c+(12(-b-1)+6)\w.
\]
Hence conjugation acts on the parameters by
\[
 (a,b)\mapsto (a,-b-1),
\]
and the entries in the tables below are listed up to this symmetry.

The search produced several additional twice semi-regular
$D(z)$-quadruples, but it did not produce examples for all
representatives in the tested range. For all representatives in the
tested range except those listed in Tables~\ref{tab:remaining_24a5}
and~\ref{tab:remaining_24a2}, the search found a  twice semi-regular $D(z)$-quadruple, up to
the above conjugation. The cases in this range for which our
regular construction search did not find a quadruple are listed in the
following tables. These tables should therefore be understood only as a
record of the limitations of this particular regular search. They do not
assert non-existence of $D(z)$-quadruples.

We emphasize that the first column records the congruence class of the
displayed representative only. The tables are not meant to say that an
entire congruence class remains unresolved. In fact, for the other
representatives in these classes within the range
$-100\leq a,b\leq 100$, explicit quadruples arising from the regular
construction were found, except for the displayed values and their
conjugates.

The corresponding candidates in the third exceptional family are not
listed separately. They are obtained from the second family by the
transformation $z\mapsto -2z$, equivalently by the congruence
transformation used earlier in the proof of Proposition~\ref{prop5_new}.

These computations suggest that the stronger assertion, namely existence
of a $D(z)$-quadruple for every $a,b$ in the exceptional families, is not
likely to follow from the regular constructions considered here. In
particular, several small representatives in the exceptional classes
resisted an extensive regular search. If $D(z)$-quadruples exist for
these values, they may have to lie outside the twice semi-regular construction, or arise from a different method.

\begin{table}[h!]
\centering
\begin{tabular}{c c c c}
\hline
$(a,b)\bmod 11$ & conjugate class & representative $(a,b)$ & $z$ \\
\hline
$(0,3)$ & $(0,7)$ & $(0,3)$ & $5+42\w$ \\
$(1,9)$ & $(1,1)$ & $(1,9)$ & $29+114\w$ \\
$(1,1)$ & $(1,9)$ & $(-54,56)$ & $-1291+678\w$ \\
$(2,1)$ & $(2,9)$ & $(46,100)$ & $1109+1206\w$ \\
$(2,9)$ & $(2,1)$ & $(57,9)$ & $1373+114\w$ \\
$(4,3)$ & $(4,7)$ & $(70,80)$ & $1685+966\w$ \\
$(6,10)$ & $(6,0)$ & $(-16,76)$ & $-379+918\w$ \\
$(6,0)$ & $(6,10)$ & $(39,0)$ & $941+6\w$ \\
$(8,6)$ & $(8,4)$ & $(-80,61)$ & $-1915+738\w$ \\
$(8,4)$ & $(8,6)$ & $(-69,26)$ & $-1651+318\w$ \\
$(8,4)$ & $(8,6)$ & $(-3,4)$ & $-67+54\w$ \\
$(8,6)$ & $(8,4)$ & $(96,83)$ & $2309+1002\w$ \\
$(10,10)$ & $(10,0)$ & $(-12,43)$ & $-283+522\w$ \\
\hline
\end{tabular}
\caption{Individual values, up to conjugation, for which the regular search did not find a  twice semi-regular $D(z)$-quadruple for $z=24a+5+(12b+6)\w$.}
\label{tab:remaining_24a5}
\end{table}

\begin{table}[h!]
\centering
\begin{tabular}{c c c c}
\hline
$(a,b)\bmod 11$ & conjugate class & representative $(a,b)$ & $z$ \\
\hline
$(0,5)$ & $(0,5)$ & $(0,16)$ & $2+198\w$ \\
$(1,0)$ & $(1,10)$ & $(-10,99)$ & $-238+1194\w$ \\
$(2,5)$ & $(2,5)$ & $(-9,27)$ & $-214+330\w$ \\
$(3,8)$ & $(3,2)$ & $(-52,85)$ & $-1246+1026\w$ \\
$(3,2)$ & $(3,8)$ & $(-30,35)$ & $-718+426\w$ \\
$(3,2)$ & $(3,8)$ & $(3,79)$ & $74+954\w$ \\
$(3,2)$ & $(3,8)$ & $(25,90)$ & $602+1086\w$ \\
$(3,8)$ & $(3,2)$ & $(69,96)$ & $1658+1158\w$ \\
$(3,5)$ & $(3,5)$ & $(-74,16)$ & $-1774+198\w$ \\
$(4,6)$ & $(4,4)$ & $(37,39)$ & $890+474\w$ \\
$(4,5)$ & $(4,5)$ & $(70,38)$ & $1682+462\w$ \\
$(4,5)$ & $(4,5)$ & $(70,49)$ & $1682+594\w$ \\
$(5,3)$ & $(5,7)$ & $(-39,36)$ & $-934+438\w$ \\
$(5,7)$ & $(5,3)$ & $(-17,51)$ & $-406+618\w$ \\
$(5,7)$ & $(5,3)$ & $(-6,18)$ & $-142+222\w$ \\
$(8,1)$ & $(8,9)$ & $(-47,67)$ & $-1126+810\w$ \\
$(9,9)$ & $(9,1)$ & $(-2,9)$ & $-46+114\w$ \\
$(9,1)$ & $(9,9)$ & $(9,1)$ & $218+18\w$ \\
$(9,9)$ & $(9,1)$ & $(42,9)$ & $1010+114\w$ \\
\hline
\end{tabular}
\caption{Individual values, up to conjugation, for which the regular search did not find a  twice semi-regular $D(z)$-quadruple for $z=24a+2+(12b+6)\w$.}
\label{tab:remaining_24a2}
\end{table}

\newpage
\subsection{A non-existence result within the regular construction} 

We now give an exact finite verification for one of the entries in
Table~\ref{tab:remaining_24a5}.  The purpose of the calculation below is to show that the
failure of the computer search is not always merely a matter of search
bounds.  For the value $z=29+114\omega$, corresponding to
$(a,b)=(1,9)$ in the first family, we prove that no proper twice
semi-regular $D(z)$-quadruple arises from the regular construction.
Thus, if a $D(z)$-quadruple exists for this value of $z$, it must lie outside the twice semi-regular construction considered here.
 Of course, this does not prove non-existence of $D(z)$-quadruples
for this value of $z$.

\begin{proposition}
Let $z=29+114\w$.  There is no proper $D(z)$-quadruple of the form $\{u,v,u+v+2r,u+4v+4r\}$, $uv+z=r^2$.
\end{proposition} 

\begin{proof}
We use the notation and the regular construction from Section~\ref{sec:prelim}.  By \eqref{eq:ef}, we choose a divisor $e\mid 3z$, put $f=3z/e$, and write
$S=e+f$.  Substituting  $r=(S-2u)/4$ from \eqref{eq:r-from-eu} into \eqref{eq:v-from-r-u}, we get
\[
        16(r^2-z)
        =
        (S-2u)^2-16z
        =
        (S^2-16z)+4u(u-S).
\]
Hence, whenever \eqref{eq:v-from-r-u} gives $v\in\Z[\w]$, we must have
$u\mid r^2-z$, and therefore $u\mid T_e$, where  $ T_e=S^2-16z$. 
Thus for each fixed divisor $e\mid 3z$, it remains only to
consider divisors $u\mid T_e$ for which both \eqref{eq:r-from-eu} and
\eqref{eq:v-from-r-u} define elements in $\Z[\w]$. The ring $\Z[\w]$ is norm-Euclidean and hence a unique factorization domain. 
In each of the factorizations below, the norms of all displayed factors other than powers of $\w$ are rational primes. Therefore 
the displayed factorizations determine all divisors $u\mid T_e$.

We have
$
        N(z)=26833 
$
and $26833$ is prime.  Since $3=(1+\w)(1-\w)$ in $\Z[\w]$, it is enough,
up to multiplication by the units $\pm1$ and up to interchanging $e$ and
$f$, to consider 
\[
        e=1,
        \quad
        e=3,
        \quad
        e=1+\w,
        \quad
        e=1-\w.
\]

First let $e=1$.  Then
$
        S=1+3z=88+342\w,
$
and
$
        T_1=
        -226648+58368\w
        =
        -8(57-7\w)(513-65\w).
$
A direct check of all divisors $u\mid T_1$, using the above
factorization and the fact that the only units in $\Z[\w]$ are $\pm1$,
shows that the only divisors for which the corresponding elements
$r$ and $v$ defined by \eqref{eq:r-from-eu} and
\eqref{eq:v-from-r-u} both belong to $\Z[\w]$ are, up to sign, 
$u=14+57\w$ and  $u=130+513\w$.  This gives 
\[
        \{14+57\w,\,
          14+57\w,\,
          58+228\w,\,
          130+513\w\},
\]
and
\[
        \{130+513\w,\,
          14+57\w,\,
          58+228\w,\,
          14+57\w\},
\]
which are not proper quadruples, respectively.  

Next let $e=3$.  Then
$
        S=3+z=32+114\w,
$
and
$
        T_3=
        -25432+5472\w
        =
        -8(57-7\w)(57-5\w).
$
Similarly, we have $u=10+57\w$ and $ u=14+57\w$. 
They yield
\[
        \{10+57\w,\,
          -2,\,
          14+57\w,\,
          14+57\w\}
\]
and
\[
        \{14+57\w,\,
          -2,\,
          14+57\w,\,
          10+57\w\}, 
\]
respectively. Thus neither case gives a proper quadruple.

If we take  $e=1+\w$, we have $f=257+85\w$ and $S=258+86\w$. Hence, 
$
        T_{1+\w}=
        51308+42552\w
        =
        -4(1-6\w)(9-2\w)(213-86\w).
$
A direct check of the divisors $u\mid T_{1+\w}$ shows that no
divisor gives $r,v \in \Z[\w]$. 

Finally, let $e=1-\w$.  Then
$f=-199+143\w$, $S=-198+142\w$,
and
$
        T_{1-\w}= 
        -1588-58056\w
        =
        4(3-2\w)(11-12\w)(243+28\w).
$ 
Again, no divisor $u\mid T_{1-\w}$ gives $r,v \in \Z[\w]$. 

The cases obtained by replacing $e$ with $-e$ give the same possibilities
up to multiplication by the units $\pm1$.  Thus all possible choices of
$e\mid 3z$ have been exhausted, up to units and up to interchanging $e$
and $f$.  In every case, either at least one of the elements $r$ and $v$ does not
belong to $\Z[\w]$, or the resulting set has a repeated entry. 
Hence no proper quadruple is obtained.

\end{proof}
This representative example shows that the twice semi-regular
constructions considered in this paper cannot, by themselves, be
expected to resolve all individual exceptional values listed in the
tables. The same type of finite divisor enumeration should nevertheless
be feasible for many of the remaining entries.  In fact, for most
entries in Table~1 the norm of $z$ is a prime. Hence, using
$3=(1+\omega)(1-\omega)$ and the fact that the only units in $\Z[\w]$ are
$\pm1$, the possible divisors $e\mid 3z$ reduce, up to units and the
interchange of $e$ and $3z/e$, to only four representative choices.
For most entries in Table~2 one has the norm $N(z)=4p$, with a 
prime $p$.  In
these cases $z=2z_0$ with $N(z_0)=p$, and since
$2=-\omega^2$, we have
$3z=-(1+\omega)(1-\omega)\omega^2z_0$.  Hence the divisor structure of
$3z$ is still very restricted. That is, up to units and the interchange of
$e$ and $3z/e$, only a small finite number of choices for $e$ has to be
checked. Thus the remaining entries appear to be amenable to analogous exact
divisor computations by the same method, although the resulting
enumerations may be longer.



\section{The individual exceptional values $-1$ and $1\pm 2\w$}\label{sec:individ}
In this section, we give some remarks on the three individual exceptional values 
$z\in\{-1,\,1\pm 2\omega\}$ which appear separately in
Theorem~\ref{tm_Dujella_Soldo}, and are not addressed by the
congruence-class constructions of the preceding sections. 
Note that $1\pm 2\w=-1\cdot (1\mp \w)^2$. 
Hence, by Lemma~\ref{lem:scaling}, the existence of a
$D(-1)$-quadruple in $\Z[\w]$ would imply the existence of
$D(1\pm 2\w)$-quadruples. The converse, however, does not necessarily
follow.

This makes the study of $D(-1)$-quadruples in $\Z[\w]$ particularly
natural, and more generally in imaginary quadratic rings of integers.
At present, no example of a $D(-1)$-quadruple in $\Z[\w]$ is known to
us. In addition, extensive computer searches have failed to find such an
example, which supports the conjecture that no $D(-1)$-quadruple exists
in this ring. This motivates the study of a more local extension
problem, i.e., for a given exact $D(-1)$-pair or $D(-1)$-triple, can it
be extended to a $D(-1)$-quadruple? In joint work with coauthors, the
second author has obtained several results in this direction, proving
the non-extendibility of certain $D(-1)$-pairs and $D(-1)$-triples in
rings $\Z[\sqrt{-k}]$, where $k>0$. Relevant references include, for
instance, \cite{DujeWeb} and further papers listed on the second author's
personal webpage. We also recall that in the classical case of $\Z$, it
was proved in \cite{BoCiMo} that no $D(-1)$-quadruple exists.

Let us briefly discuss the remaining values $z=1\pm 2\w$. By
conjugation, as in Lemma~\ref{lem:scaling}, the existence of a
$D(1+2\w)$-quadruple would imply the existence of a
$D(1-2\w)$-quadruple. In our computer search, however, we were unable to
find examples of such $D(z)$-quadruples. Nevertheless, $D(z)$-triples
do exist. For instance, the set
$
\{-2,1,1\pm 2\w\}
$ 
is a $D(1\pm 2\w)$-triple.

We also note that there is no $D(z)$-triple containing two positive
integers. Indeed, suppose that $u,v$ are positive integers and that
$ 
uv+z=r^2,
$ 
for some $r=m+n\w \in \Z[\w]$. Since
\[
(m+n\w)^2=m^2-2n^2+2mn\w,
\]
comparison of coefficients gives
\[
uv+1=m^2-2n^2,\quad mn=\pm 1.
\]
This yields $uv=-2$, which is impossible for positive integers $u$ and
$v$.

There are also infinite families of $D(z)$-triples. For example, the set
\[
S_n=\{1,n^2-1\mp 2\w,n^2+2n\mp 2\w\},\quad n\in \Z,
\]
is a $D(z)$-triple. We can show that none of these triples can be
extended to a $D(z)$-quadruple.

A direct computation shows that the only squares modulo $4$ in
$\Z[\w]$ are
$
0, 1, 2, 3+2\w.
$ 
Suppose that there exists $d\in \Z[\w]$ such that all three elements
\[
d+z,\quad (n^2-1\mp 2\w)d+z,\quad (n^2+2n\mp 2\w)d+z
\]
are squares in $\Z[\w]$. Since $z=1\pm 2\w\equiv 1+2\w \pmod 4$, the
condition that $d+z$ be a square modulo $4$ implies that
\[
d\equiv 3+2\w, 2\w, 1+2\w, 2 \pmod 4,
\]
respectively.

On the other hand, according to the parity of $n$, the two elements
$n^2-1\mp 2\w$, $n^2+2n\mp 2\w$ 
are congruent modulo $4$ to
$2\w$ and $3+2\w$
in some order. Hence, for the four possible congruence classes of $d$
modulo $4$, we obtain
\[
\bigl(2\w d+z,\,(3+2\w)d+z\bigr)
\equiv
(1,2+2\w),\,\,
(1+2\w,1),\,\,
(1,2\w),\,\,
(1+2\w,3+2\w)
\pmod 4,
\]
respectively. However, the elements
$1+2\w$, $2\w$, $2+2\w$ 
are not squares modulo $4$. This gives a contradiction in each of the
four cases. Therefore, the triple $S_n$ cannot be extended to a
$D(z)$-quadruple.

\section*{Acknowledgement} 
The authors were supported by the Croatian Science Foundation under the
project no. IP-2022-10-5008 (TEBAG). The first author acknowledge support from the
project ``Implementation of cutting-edge research and its application as part of the
Scientific Center of Excellence for Quantum and Complex Systems, and Representations 
of Lie Algebras'', Grant No. PK.1.1.10.0004, co-financed by the European
Union through the European Regional Development Fund – Competitiveness and
Cohesion Programme 2021-2027.  This research was funded by the European union: 
NextGenerationEU through the National Recovery and Resilience Plan 2021-2026 under the project 
Advanced Algorithms and Optimization Models Supported by Mathematical Theory – OptimaAI (581-UNIOS-54) and   
Institutional grant of University of Zagreb, Faculty of Science (IK IA 1.1.3. Impact4Math). 

\bibliographystyle{plain}

\noindent
{\sc
Andrej Dujella\\
Department of Mathematics\\
University of Zagreb\\
Bijeni\v{c}ka cesta 30\\
HR-10000 Zagreb, Croatia\\
e-mail: {\rm duje@math.hr}

\bigskip
\noindent
Ivan Soldo\\
School of Applied Mathematics and Informatics\\
University of Osijek\\
Trg Ljudevita Gaja 6\\
HR-31000 Osijek, Croatia\\
e-mail: {\rm isoldo@mathos.hr}
}

\end{document}